\documentclass[12 pt]{amsart}
\usepackage{amsmath, amsthm, amscd, amsfonts, amssymb, graphicx, color}
\usepackage[ colorlinks, plainpages]{hyperref}
\usepackage{tikz}
\usepackage{pgfplots}
\usetikzlibrary{automata}
\newtheorem{theorem}{Theorem}[section]
\newtheorem{lemma}[theorem]{Lemma}
\newtheorem{proposition}[theorem]{Proposition}
\newtheorem{corollary}[theorem]{Corollary}
\theoremstyle{definition}
\newtheorem{definition}[theorem]{Definition}
\newtheorem{example}[theorem]{Example}

\theoremstyle{remark}
\newtheorem{remark}[theorem]{Remark}
\numberwithin{equation}{section}
\begin{document}
\title[]{Edge Ideals of Prime Ideal Graphs over Finite Rings: Ordinary Powers, Fiber Cones, and Linear Powers}
\author[]{Tabinda Rasheed$^1$, Wang Yao$^2$ }
\maketitle
{\center{$^{1,2}$ School of Mathematics and Statistics, \\Nanjing University of Information Science and Technology, Nanjing 210044, P.R. China.}\\
\footnotesize{Corresponding author. Email address:
\href{mailto:xx@xx}{tabindarasheed00@gmail.com}, wangyao@nuist.edu.cn} \\ {}}

\begin{abstract}
Let \(R\) be a finite commutative ring with identity and let \(P\) be a proper prime ideal of \(R\). The prime ideal graph \(\Gamma_P(R)\) has vertex set \(R\setminus\{0\}\), where two distinct vertices \(x\) and \(y\) are adjacent if and only if \(xy\in P\). We prove that prime ideal graphs form a ring-realizable subfamily of complete split graphs. More precisely, if $m=|P|$, $q=|R/P|$,
then \(q\) is a prime power and
$\Gamma_P(R)\cong K_{m-1}\vee \overline{K}_{m(q-1)}$.
We also prove a realization theorem showing that every complete split graph of this form arises from a prime ideal of a finite commutative ring. For the edge ideal \(I=I(\Gamma_P(R))\), we determine the minimal vertex covers and obtain the irredundant primary decomposition. We characterize the minimal monomial generators of every ordinary power \(I^n\) and derive a closed formula for \(\mu(I^n)\). We further interpret this formula as the Hilbert function of the special fiber ring \(\mathcal{F}(I)\), compute the analytic spread, and prove that \(\mathcal{F}(I)\) is a normal Cohen--Macaulay affine semigroup ring. Finally, we show that \(I\) is matroidal and that every ordinary power \(I^n\) is polymatroidal; consequently, \(I^n\) has linear quotients and a \(2n\)-linear minimal free resolution for all \(n\geq 1\).
\end{abstract}



\section{Introduction}

Edge ideals provide an effective algebraic framework for studying finite simple graphs through monomial ideals. If \(G\) is a finite simple graph on the vertex set $V(G)=\{z_1,\ldots,z_r\}$, then the edge ideal of \(G\) is the square-free quadratic monomial ideal
\[
I(G)=(z_i z_j:\{z_i,z_j\}\in E(G))\subseteq k[z_1,\ldots,z_r].
\]
Since the foundational work of Villarreal and of Simis, Vasconcelos, and Villarreal on the ideal theory of graphs, edge ideals have become central objects in combinatorial commutative algebra ~\cite{VIL,SIM,VILL}. Their algebraic invariants often reflect important combinatorial properties of the underlying graph. For example, minimal vertex covers correspond to minimal primes of \(I(G)\), while homological invariants such as regularity, projective dimension, Betti numbers, and linear resolutions are closely related to the structure of \(G\) ~\cite{FRO,MOR,HERZOG,DAO}. Recent work has continued to study algebraic and homological properties of edge ideals and cover ideals for special graph families; for instance, edge ideals and cover ideals of unbalanced crown graphs were investigated in ~\cite{RAT}.

A major direction in the study of edge ideals concerns the behavior of their ordinary powers. Although \(I(G)\) is generated by square-free quadratic monomials, the higher powers \(I(G)^n\) may have complicated minimal generating sets and homological behavior. Understanding the minimal generators \(G(I(G)^n)\), the generator count \(\mu(I(G)^n)\), and the regularity of \(I(G)^n\) is therefore a natural problem in combinatorial commutative algebra. Several works have studied ordinary powers of edge ideals from the viewpoints of regularity, associated primes, and linear resolutions ~\cite{HERZOG, FRAN, GU}. In particular, graph families for which all ordinary powers admit explicit generator descriptions and linear resolutions provide useful test cases where combinatorial and homological phenomena can be computed exactly.
Polymatroidal ideals form one of the most important classes of monomial ideals with good homological behavior. The theory of discrete polymatroids developed by Herzog and Hibi shows that polymatroidal ideals satisfy a strong exchange property, which leads to linear quotients and linear resolutions under suitable hypotheses ~\cite{HER}. Ideals with linear quotients have also been studied through mapping cone techniques and related homological methods ~\cite{HERZO, JAH}. Moreover, products and powers of polymatroidal ideals often retain favorable algebraic properties, making polymatroidality a useful tool for proving linear powers ~\cite{CON, BAN}. Stability properties of powers of polymatroidal ideals, including behavior related to associated primes and depth, have also been studied recently ~\cite{MAF}. Thus, when an edge ideal or its powers can be shown to be polymatroidal, one obtains not only a combinatorial description of the generators but also strong homological consequences.
In parallel, edge rings and special fiber rings provide another important viewpoint on monomial ideals. If \(I\subseteq S\) is a homogeneous ideal generated in one degree and \(\mathfrak m\) is the homogeneous maximal ideal of \(S\), then the special fiber ring
\[
\mathcal{F}(I)=\bigoplus_{n\geq 0}I^n/\mathfrak m I^n
\]
encodes the asymptotic behavior of the minimal generators of the powers of \(I\). In particular, $\operatorname{HF}_{\mathcal{F}(I)}(n)=\mu(I^n)$.
For edge ideals, the special fiber ring is closely related to the edge ring of the underlying graph. Normality and Cohen-Macaulayness of edge rings have been studied through affine semigroup methods and graph-theoretic criteria such as the odd-cycle condition ~\cite{OH,VILL}. Related Cohen-Macaulay and Gorenstein properties associated with graph complexes have also appeared in recent work ~\cite{NIK}. Therefore, computing the special fiber ring and the analytic spread of an edge ideal gives structural information beyond the minimal generators of its ordinary powers.

The present paper studies edge ideals arising from prime ideal graphs of finite commutative rings. Prime ideal graphs were introduced by Salih and Jund ~\cite{SAL}. If \(R\) is a finite commutative ring with identity and \(P\) is a proper prime ideal of \(R\), then the prime ideal graph \(\Gamma_P(R)\) has vertex set \(R\setminus\{0\}\), and two distinct vertices \(x\) and \(y\) are adjacent if and only if $xy\in P$.
Subsequent work studied graph-theoretic properties of these graphs, including connectedness, diameter, clique number, and chromatic number ~\cite{KUR}. More broadly, graph constructions arising from finite rings continue to attract attention, including recent work on zero-divisor graphs and their graph invariants ~\cite{DO}. These studies show that graphs associated with finite rings form a natural bridge between ring theory and graph theory. However, the edge ideals associated with prime ideal graphs, their ordinary powers, and their fiber-cone invariants have not been studied in comparable algebraic detail.

The first contribution of this paper is a ring-theoretic classification of the complete split graphs arising from prime ideal graphs. Let $m=|P|$, $q=|R/P|$. Since \(P\) is prime and \(R\) is finite, the quotient \(R/P\) is a finite field; hence \(q\) is a prime power and \(|R|=mq\). We prove that $\Gamma_P(R)\cong K_{m-1}\vee \overline{K}_{m(q-1)}$.
Thus, the complete split graphs arising from prime ideals of finite rings are not arbitrary; their parameters are determined by the finite-ring pair \((|P|,|R/P|)\). We also prove a realization theorem showing that every complete split graph of the form $K_{m-1}\vee \overline{K}_{m(q-1)}$,
where \(q\) is a prime power, occurs as the prime ideal graph of a finite commutative ring. Therefore, the family considered in this paper is a ring-realizable subfamily of complete split graphs whose parameters are determined by finite-ring data.
The distinction from arbitrary complete split graphs is important. A general complete split graph \(K_a\vee \overline{K}_b\) has two independent graph parameters \(a\) and \(b\). In contrast, a prime ideal graph of a finite commutative ring satisfies $a=|P|-1$, $b=|P|(|R/P|-1)$. Equivalently, $b=(a+1)(q-1)$,
where \(q=|R/P|\) is a prime power. Hence the finite-ring origin imposes an arithmetic constraint on the split-graph parameters. Complete split graphs form a basic and tractable class in graph theory, lying within the broader family of split and threshold-type graphs. The finite-ring construction considered here therefore provides a natural algebraic laboratory in which ordinary powers, primary decompositions, special fiber rings, analytic spread, and linear resolutions of the associated edge ideals can be computed explicitly in terms of the ring-theoretic pair \((|P|,|R/P|)\).
Using this classification, we study the edge ideal \(I=I(\Gamma_P(R))\). Writing $a=|P|-1$, $b=|P|(|R/P|-1)$,
we work in the polynomial ring \(S=k[x_1,\ldots,x_a,y_1,\ldots,y_b]\). Since the vertices corresponding to \(P\setminus\{0\}\) form a clique, the vertices corresponding to \(R\setminus P\) form an independent set, and every clique vertex is adjacent to every independent vertex, the edge ideal is
\[
I=(x_i x_j,\;x_i y_\ell
\mid 1\leq i<j\leq a,\;1\leq i\leq a,\;1\leq \ell\leq b).
\]
We determine the minimal vertex covers of \(\Gamma_P(R)\) and obtain an irredundant primary decomposition of \(I\). In particular, the number of minimal primes is exactly \(|P|\), giving a direct ring-theoretic interpretation of the primary decomposition.
We then describe the minimal monomial generators of every ordinary power \(I^n\). More precisely, for every \(n\geq 1\), we prove that
\[
G(I^n)=
\left\{
x^\alpha y^\beta:
|\alpha|+|\beta|=2n,\;
|\beta|\leq n,\;
0\leq \alpha_i\leq n
\text{ for all }i
\right\}.
\]
This yields a closed formula for \(\mu(I^n)\) in terms of \(n\), \(|P|\), and \(|R/P|\). We further interpret this formula as the Hilbert function of the special fiber ring \(\mathcal{F}(I)\), compute the analytic spread, and prove that \(\mathcal{F}(I)\) is a normal Cohen--Macaulay affine semigroup ring.
Finally, we show that \(I\) is matroidal and hence polymatroidal. Consequently, every ordinary power \(I^n\) is polymatroidal, has linear quotients, and has a \(2n\)-linear minimal free resolution. We also record a parameter-invariance result showing that the ordinary-power generator counts, special-fiber Hilbert functions, analytic spread, and regularity sequence depend only on the pair \((|P|,|R/P|)\), not on the particular presentation of the ring.

The results of this paper lie at the intersection of finite ring theory, graph theory, and combinatorial commutative algebra. The prime ideal graph construction provides a finite-ring source of complete split graphs whose parameters are constrained by residue-field and prime-ideal data. From this input, we obtain explicit primary decompositions, ordinary-power formulas, fiber-cone Hilbert functions, analytic spread, normality, Cohen-Macaulayness, and linear-resolution results for the corresponding edge ideals. The novelty of this paper is twofold. First, prime ideal graphs are shown to form an arithmetically constrained, ring-realizable subfamily of complete split graphs, rather than an arbitrary split-graph class. Second, for this ring-induced family, we give a uniform algebraic analysis of the associated edge ideals, including primary decompositions, ordinary power generators, fiber cone Hilbert functions, analytic spread, normality, Cohen-Macaulayness, polymatroidality, and linear powers. These results are expressed directly in terms of the finite-ring parameters \(|P|\) and \(|R/P|\).

The paper is organized as follows. In Section~2, we classify prime ideal graphs as ring-realizable complete split graphs and prove the realization theorem. In Section~3, we determine the minimal vertex covers and the primary decomposition of the edge ideal. Section~4 describes the minimal generators of all ordinary powers and derives the formula for \(\mu(I^n)\). Section~5 studies the special fiber ring, analytic spread, normality, and Cohen-Macaulayness. Section~6 proves polymatroidality, linear quotients, linear resolutions, and parameter invariance. Section~7 gives examples, and the final section concludes the paper.

\section{Ring-induced complete split graphs}

In this section, we refine the graph-theoretic reduction of prime ideal graphs. The purpose is to show that the complete split graphs arising from prime ideals of
finite rings are not arbitrary complete split graphs. Their parameters are constrained by the quotient field \(R/P\). This gives a ring-theoretic classification and a realization result for the graph family considered in the paper.

Throughout this section, \(R\) denotes a finite commutative ring with identity and
\(P\) denotes a proper prime ideal of \(R\). The prime ideal graph \(\Gamma_P(R)\)
has vertex set \(R\setminus\{0\}\), and two distinct vertices \(x,y\in R\setminus\{0\}\)
are adjacent if and only if \(xy\in P\).
For two vertex-disjoint graphs \(G\) and \(H\), their join \(G\vee H\) is obtained
from the disjoint union of \(G\) and \(H\) by adding all edges between every vertex
of \(G\) and every vertex of \(H\). We write \(\overline{K}_b\) for the edgeless graph
on \(b\) vertices.

\begin{proposition}\label{prop:ring-induced-split}
Let \(R\) be a finite commutative ring with identity and let \(P\) be a proper prime ideal of \(R\). Set $m=|P|$, $q=|R/P|$. Then \(q\) is a prime power and $|R|=mq$. Moreover,
\[
\Gamma_P(R)\cong K_{m-1}\vee \overline{K}_{m(q-1)}.
\]
Equivalently,
\[
\Gamma_P(R)\cong K_{|P|-1}\vee \overline{K}_{|P|(|R/P|-1)}.
\]
\end{proposition}

\begin{proof}
Since \(P\) is a prime ideal of \(R\), the quotient ring \(R/P\) is an integral domain. Since \(R\) is finite, \(R/P\) is a finite integral domain, hence a finite field. Therefore \(q=|R/P|\) is a prime power. The quotient map $\pi:R\longrightarrow R/P$ has kernel \(P\), and each coset of \(P\) has exactly \(|P|=m\) elements. Hence
$|R|=|P|\cdot |R/P|=mq$. Now set
$A=P\setminus\{0\}$, $B=R\setminus P$.
Then
\[
|A|=|P|-1=m-1\quad \text{and}\quad |B|=|R|-|P|=mq-m=m(q-1).
\]
We show that \(A\) is a clique, \(B\) is an independent set, and every vertex in \(A\) is adjacent to every vertex in \(B\). If \(a,a'\in A\) are distinct, then \(a,a'\in P\). Since \(P\) is an ideal, we have \(aa'\in P\). Therefore, every two distinct vertices of \(A\) are adjacent, so \(A\) is a clique. If \(a\in A\) and \(b\in B\), then \(a\in P\) and \(b\in R\). Again, since \(P\) is an ideal, \(ab\in P\). Hence, every vertex of \(A\) is adjacent to every vertex of
\(B\).

Finally, let \(b,b'\in B\). If \(bb'\in P\), then the primality of \(P\) implies
\(b\in P\) or \(b'\in P\), which contradicts \(b,b'\in R\setminus P\). Hence
\(bb'\notin P\), and so no two vertices of \(B\) are adjacent. Thus \(B\) is an
independent set. Therefore \(\Gamma_P(R)\) is the join of a complete graph on \(m-1\) vertices and
an edgeless graph on \(m(q-1)\) vertices:
$\Gamma_P(R)\cong K_{m-1}\vee \overline{K}_{m(q-1)}$.
This proves the assertion.
\end{proof}

\begin{remark}\label{rem:finite-field-case}
If \(P=(0)\), then \(m=|P|=1\). In this case \(R/P=R\) is a finite field, and $\Gamma_P(R)\cong \overline{K}_{q-1}$.
Thus, the associated edge ideal is the zero ideal. In the remaining sections, when we study vertex covers, primary decompositions, and powers of the edge ideal, we
will assume \(m\geq 2\), equivalently \(|P|\geq 2\), so that the graph has at least
one clique vertex.
\end{remark}

\begin{corollary}\label{cor:ring-parameter-restriction}
Let \(a,b\) be nonnegative integers with \(b\geq 1\). If a complete split graph $K_a\vee \overline{K}_b$
is isomorphic to \(\Gamma_P(R)\) for some finite commutative ring \(R\) and some
proper prime ideal \(P\), then $a=|P|-1$ and $b=(a+1)(q-1)$ for some prime power \(q\). In particular,
\[
q=\frac{b}{a+1}+1
\]
must be a prime power. Hence not every complete split graph is induced by a prime
ideal of a finite ring.
\end{corollary}

\begin{proof}
By Proposition~\ref{prop:ring-induced-split}, every prime ideal graph has the form
\[
\Gamma_P(R)\cong K_{|P|-1}\vee \overline{K}_{|P|(|R/P|-1)}.
\]
Thus, if \(K_a\vee \overline{K}_b\cong \Gamma_P(R)\), then $a=|P|-1$ and $b=|P|(|R/P|-1)=(a+1)(q-1)$, where \(q=|R/P|\). Since \(R/P\) is a finite field, \(q\) is a prime power. This
proves the claim.
\end{proof}

\begin{theorem}\label{thm:realization}
Let \(q\) be a prime power and let \(m\geq 1\) be an integer. Then there exists
a finite commutative ring \(R\) with identity and a proper prime ideal \(P\) of
\(R\) such that
\[
|P|=m
\qquad\text{and}\qquad
|R/P|=q.
\]
Consequently,
\[
\Gamma_P(R)\cong K_{m-1}\vee \overline{K}_{m(q-1)}.
\]
\end{theorem}

\begin{proof}
We consider two cases.
First suppose \(m=1\). Let $R=\mathbb{F}_q$ and $P=(0)$.
Then \(R\) is a finite field, so \(P=(0)\) is a proper prime ideal of \(R\). Moreover, $|P|=1=m$ and $R/P\cong \mathbb{F}_q$,
so \(|R/P|=q\).
Now suppose \(m\geq 2\). Let
\[
R=\mathbb{F}_q\times \mathbb{Z}/m\mathbb{Z}
\]
and let
\[
P=\{0\}\times \mathbb{Z}/m\mathbb{Z}.
\]
Then \(R\) is a finite commutative ring with identity, and \(P\) is a proper ideal of
\(R\). Also,
\[
|P|=|\mathbb{Z}/m\mathbb{Z}|=m.
\]
Furthermore, $R/P\cong \mathbb{F}_q$.
Since \(\mathbb{F}_q\) is a field, \(R/P\) is an integral domain. Hence \(P\) is a
prime ideal of \(R\), and $|R/P|=q$.
In both cases, we have constructed a finite commutative ring \(R\) and a proper
prime ideal \(P\) such that \(|P|=m\) and \(|R/P|=q\). Therefore, by
Proposition~\ref{prop:ring-induced-split},
\[
\Gamma_P(R)\cong K_{m-1}\vee \overline{K}_{m(q-1)}.
\]
This proves the theorem.
\end{proof}

\begin{corollary}\label{cor:classification}
A complete split graph \(K_a\vee \overline{K}_b\) with \(a\geq 0\) and \(b\geq 1\)
arises as a prime ideal graph of a finite commutative ring if and only if there exists
a prime power \(q\) such that $b=(a+1)(q-1)$. In that case, one may realize the graph by taking $R=\mathbb{F}_q\times \mathbb{Z}/(a+1)\mathbb{Z}$ and $P=\{0\}\times \mathbb{Z}/(a+1)\mathbb{Z}$.
\end{corollary}

\begin{proof}
The necessity follows from Corollary~\ref{cor:ring-parameter-restriction}. Conversely,
suppose that $b=(a+1)(q-1)$ for some prime power \(q\). Set \(m=a+1\). By Theorem~\ref{thm:realization}, there
exists a finite commutative ring \(R\) and a prime ideal \(P\) such that $|P|=m=a+1$, $|R/P|=q$.
Therefore,
$\Gamma_P(R)\cong K_{m-1}\vee \overline{K}_{m(q-1)}
=K_a\vee \overline{K}_b$.
This proves the equivalence.
\end{proof}
\begin{remark}
The classification contains several familiar graph families as special cases. If
\(|P|=2\), then
\[
\Gamma_P(R)\cong K_1\vee \overline{K}_{2(|R/P|-1)},
\]
which is a star graph. If \(|R/P|=2\), then
\[
\Gamma_P(R)\cong K_{|P|-1}\vee \overline{K}_{|P|}.
\]
Thus the ring-induced family includes stars and an infinite class of complete split
graphs determined by finite-field residue sizes.
\end{remark}

\begin{remark}\label{rem:scope-distinction}
The preceding classification shows that the graphs studied in this paper form a
ring-induced subfamily of complete split graphs whose parameters are governed by
\(|P|\) and \(|R/P|\). Thus, the focus of the paper is not the general theory of
complete split graphs, but rather the ordinary powers and fiber-cone
invariants of edge ideals arising from prime ideals of finite commutative rings.
This distinction will be used throughout the paper to express the algebraic
invariants in terms of finite-ring data.
\end{remark}

\section{Vertex covers and primary decomposition}

In this section, we determine the minimal vertex covers and the primary decomposition
of the edge ideal of a prime ideal graph. We keep the ring-theoretic notation from
the previous section. Thus \(R\) is a finite commutative ring with identity, \(P\) is
a proper prime ideal of \(R\), and $m=|P|$, $q=|R/P|$.

By Proposition~\ref{prop:ring-induced-split}, $\Gamma_P(R)\cong K_{m-1}\vee \overline{K}_{m(q-1)}$.
We set $a=m-1$, $b=m(q-1)$.
When \(m=1\), we have \(P=(0)\), the graph is edgeless, and its edge ideal is the
zero ideal. Hence, throughout the rest of the paper, whenever we discuss vertex
covers, primary decompositions, and powers of the edge ideal, we assume $m\geq 2$. Equivalently, \(a=m-1\geq 1\).

Let
\[
A=P\setminus\{0\}=\{u_1,\ldots,u_a\}\quad \text{and}\quad  B=R\setminus P=\{v_1,\ldots,v_b\}.
\]
We attach the variables \(x_1,\ldots,x_a\) to the vertices \(u_1,\ldots,u_a\) and
the variables \(y_1,\ldots,y_b\) to the vertices \(v_1,\ldots,v_b\). 

Let $S=k[x_1,\ldots,x_a,y_1,\ldots,y_b]$ be a polynomial ring over a field \(k\). By the structure of \(\Gamma_P(R)\), the
vertices in \(A\) form a clique, the vertices in \(B\) form an independent set, and
every vertex of \(A\) is adjacent to every vertex of \(B\). Therefore the edge ideal
of \(\Gamma_P(R)\) is
\[
I=I(\Gamma_P(R))
=
(x_ix_j,\;x_i y_\ell
\mid 1\leq i<j\leq a,\;1\leq i\leq a,\;1\leq \ell\leq b)
\subseteq S.
\]
Equivalently, if
\[
J=(x_ix_j\mid 1\leq i<j\leq a) \ \
\text{and} \ \
K=(x_i y_\ell\mid 1\leq i\leq a,\;1\leq \ell\leq b),
\]
then $I=J+K$.

\begin{proposition}\label{prop:minimal-vertex-covers-ring}
Assume \(m\geq 2\). The minimal vertex covers of \(\Gamma_P(R)\) are precisely $A$
and
\[
(A\setminus\{u_i\})\cup B,\qquad 1\leq i\leq a.
\]
Equivalently, \(\Gamma_P(R)\) has exactly \(a+1=m\) minimal vertex covers.
\end{proposition}

\begin{proof}
Since every edge of \(\Gamma_P(R)\) has at least one endpoint in \(A\), the set
\(A\) is a vertex cover. It is minimal because if \(u_i\in A\) is removed, then the
edge \(u_i v_1\) is no longer covered.

Now fix \(i\in\{1,\ldots,a\}\). We claim that
\[
(A\setminus\{u_i\})\cup B
\]
is a minimal vertex cover. Every edge inside the clique \(A\) has both endpoints in
\(A\), and hence every such edge is covered by \(A\setminus\{u_i\}\), except for
edges incident with \(u_i\). But if an edge has the form \(u_i u_j\), with \(j\neq i\),
then it is covered by \(u_j\in A\setminus\{u_i\}\). Every edge between \(A\) and
\(B\) is also covered: if it is incident with a vertex in \(A\setminus\{u_i\}\), then
it is covered by that vertex; if it is of the form \(u_i v_\ell\), then it is covered
by \(v_\ell\in B\). Thus \((A\setminus\{u_i\})\cup B\) is a vertex cover.
It remains to check minimality. If one removes a vertex \(u_j\in A\setminus\{u_i\}\),
then the edge \(u_i u_j\) is uncovered. If one removes a vertex \(v_\ell\in B\), then
the edge \(u_i v_\ell\) is uncovered. Hence the cover is minimal.

Conversely, let \(C\) be a minimal vertex cover of \(\Gamma_P(R)\). If there exists
a vertex \(v_\ell\in B\) such that \(v_\ell\notin C\), then every vertex of \(A\)
must belong to \(C\), because \(v_\ell\) is adjacent to every vertex of \(A\). Hence
\(A\subseteq C\). By minimality of \(C\), we get \(C=A\).
Now suppose \(B\subseteq C\). Since \(A\) is a clique, the cover \(C\) must contain
at least \(a-1\) vertices of \(A\). If it contained all vertices of \(A\), then \(C\)
would not be minimal because any vertex of \(B\) could be removed while the set
\(A\) would still cover all edges. Therefore \(C\) contains exactly \(a-1\) vertices
of \(A\). Hence
\[
C=(A\setminus\{u_i\})\cup B
\]
for some \(i\). This proves the claim.
\end{proof}

\begin{corollary}\label{cor:height-dimension}
Assume \(m\geq 2\). Then the height of \(I\) is $\operatorname{ht}(I)=a=m-1$. Consequently, $\dim(S/I)=b=m(q-1)$.
\end{corollary}

\begin{proof}
The height of a square-free edge ideal is the minimum cardinality of a vertex cover
of the corresponding graph. By Proposition~\ref{prop:minimal-vertex-covers-ring},
one minimal vertex cover is \(A\), which has cardinality $|A|=a=m-1$. The other minimal vertex covers have cardinality
\[|(A\setminus\{u_i\})\cup B|=a-1+b.
\]
Since \(b\geq 1\), we have $a\leq a-1+b$. Therefore \(\operatorname{ht}(I)=a=m-1\). Since \(S\) has \(a+b\) variables, we get $\dim(S/I)=(a+b)-a=b=m(q-1)$.
\end{proof}

\begin{corollary}\label{cor:primary-decomposition-ring}
Assume \(m\geq 2\). Let $\mathfrak{p}_0=(x_1,\ldots,x_a)$ and, for \(1\leq i\leq a\), let
\[
\mathfrak{p}_i=(x_1,\ldots,\widehat{x_i},\ldots,x_a,y_1,\ldots,y_b),
\]
where the hat means that \(x_i\) is omitted. Then $I=\mathfrak{p}_0\cap\mathfrak{p}_1\cap\cdots\cap\mathfrak{p}_a$. In particular, \(I\) is radical and $\operatorname{Ass}(S/I)=
\{\mathfrak{p}_0,\mathfrak{p}_1,\ldots,\mathfrak{p}_a\}$.
Thus \(S/I\) has exactly \(m=|P|\) minimal primes.
\end{corollary}

\begin{proof}
For a square-free monomial edge ideal, the minimal primes correspond exactly to the minimal vertex covers of the graph. By Proposition~\ref{prop:minimal-vertex-covers-ring}, the minimal vertex covers are $A$ and $(A\setminus\{u_i\})\cup B, \ 1\leq i\leq a$. The cover \(A\) corresponds to the monomial prime ideal $\mathfrak{p}_0=(x_1,\ldots,x_a)$. For each \(i\), the cover \((A\setminus\{u_i\})\cup B\) corresponds to $\mathfrak{p}_i=(x_1,\ldots,\widehat{x_i},\ldots,x_a,y_1,\ldots,y_b)$. Hence $I=\mathfrak{p}_0\cap\mathfrak{p}_1\cap\cdots\cap\mathfrak{p}_a$.
Since \(I\) is a square-free monomial ideal, it is radical. The displayed decomposition is irredundant because the corresponding vertex covers are distinct and minimal. Therefore, the associated primes of \(S/I\) are precisely $\mathfrak{p}_0,\mathfrak{p}_1,\ldots,\mathfrak{p}_a$. There are \(a+1=m\) such primes.
\end{proof}

\begin{remark}\label{rem:non-equidimensional}
If \(b\geq 2\), then \(S/I\) is not equidimensional. Indeed, $\operatorname{ht}(\mathfrak{p}_0)=a$, whereas
$\operatorname{ht}(\mathfrak{p}_i)=a-1+b$ for \(1\leq i\leq a\). Thus, the minimal primes have different heights whenever
\(b\geq 2\). In terms of the ring parameters, this means that \(S/I\) is not
equidimensional whenever $m(q-1)\geq 2$. Consequently, when \(b\geq 2\), the quotient \(S/I\) is not Cohen-Macaulay, since
a Cohen-Macaulay ring is equidimensional. This should be distinguished from the
special fiber ring \(\mathcal{F}(I)\), which is shown in Section~5 to be normal and
Cohen-Macaulay.

\end{remark}

\begin{remark}\label{rem:ring-interpretation-primary}
Corollary~\ref{cor:primary-decomposition-ring} expresses the primary decomposition
of the edge ideal entirely in terms of the prime ideal \(P\). The prime
\(\mathfrak{p}_0\) corresponds to the cover \(P\setminus\{0\}\), while the remaining
minimal primes correspond to removing one nonzero element of \(P\) and adding all
elements outside \(P\). Hence the number of minimal primes of \(I(\Gamma_P(R))\)
is exactly \(|P|\).
\end{remark} 

\section{Ordinary powers and minimal generators}
In this section, we describe the minimal monomial generators of every ordinary power
of the edge ideal of a prime ideal graph. We continue to use the notation $m=|P|$, $q=|R/P|$,
and $a=m-1$, $b=m(q-1)$.
Thus $\Gamma_P(R)\cong K_a\vee \overline{K}_b$.
We assume \(m\geq 2\), so \(a\geq 1\). Let $S=k[x_1,\ldots,x_a,y_1,\ldots,y_b]$,
\[
I=I(\Gamma_P(R))
=
(x_ix_j,\;x_i y_\ell
\mid 1\leq i<j\leq a,\;1\leq i\leq a,\;1\leq \ell\leq b).
\]
Write
\[
J=(x_ix_j\mid 1\leq i<j\leq a) \ \text{and} \ K=(x_i y_\ell\mid 1\leq i\leq a,\;1\leq \ell\leq b).
\]
Then $I=J+K$.

For multi-indices
\[
\alpha=(\alpha_1,\ldots,\alpha_a)\in \mathbb{N}^a,
\qquad
\beta=(\beta_1,\ldots,\beta_b)\in \mathbb{N}^b,
\]
we write
\[
x^\alpha=x_1^{\alpha_1}\cdots x_a^{\alpha_a},
\qquad
y^\beta=y_1^{\beta_1}\cdots y_b^{\beta_b},
\]
and
\[
|\alpha|=\alpha_1+\cdots+\alpha_a,
\qquad
|\beta|=\beta_1+\cdots+\beta_b.
\]

\begin{lemma}\label{lem:powers-clique-part}
For every integer \(r\geq 0\),
\[
G(J^r)=
\{x^\delta: |\delta|=2r,\;0\leq \delta_i\leq r\text{ for all }i=1,\ldots,a\},
\]
where \(J^0=S\) and \(G(J^0)=\{1\}\). If \(a=1\) and \(r\geq 1\), both sides are
empty.
\end{lemma}

\begin{proof}
For \(r=0\), the assertion is clear. Assume \(r\geq 1\). If \(a=1\), then \(J=0\),
and there is no vector \(\delta\in \mathbb{N}\) satisfying
\[
|\delta|=2r
\quad\text{and}\quad
\delta_1\leq r.
\]
Thus, both sides are empty.
Now assume \(a\geq 2\). Every generator of \(J^r\) is a product of \(r\) quadratic
monomials \(x_i x_j\) with \(i\neq j\). Hence, every minimal generator has total
degree \(2r\). Moreover, a fixed variable \(x_i\) can appear at most once in each
quadratic factor, so its exponent is at most \(r\). Therefore, every monomial in
\(G(J^r)\) satisfies
\[
|\delta|=2r,\qquad 0\leq \delta_i\leq r\quad \text{for all}\quad i.
\]

Conversely, let \(x^\delta\) be a monomial satisfying
$|\delta|=2r$ and $0\leq \delta_i\leq r$
for all \(i\). We show that \(x^\delta\in J^r\). The vector \(\delta\) may be viewed
as a degree sequence on the vertex set \(\{1,\ldots,a\}\). Since the total degree is
\(2r\) and no entry exceeds \(r\), one can pair the multiset containing \(\delta_i\)
copies of \(i\) into \(r\) unordered pairs of distinct indices. Indeed, at each step, choose two indices with positive remaining multiplicity. This is always possible
because no remaining multiplicity can exceed half of the remaining total. Each pair
\(\{i,j\}\) corresponds to the edge monomial \(x_i x_j\) of the complete graph on
the \(x\)-vertices. Multiplying the \(r\) corresponding edge monomials gives
\(x^\delta\). Hence \(x^\delta\in J^r\). Since \(J^r\) is generated in degree \(2r\), no monomial of degree \(2r\) properly
divides another monomial of degree \(2r\). Therefore every such \(x^\delta\) is a
minimal generator of \(J^r\). This proves the claim.
\end{proof}

\begin{lemma}\label{lem:powers-cross-part}
For every integer \(s\geq 0\), $G(K^s)=
\{x^\gamma y^\beta: |\gamma|=s,\;|\beta|=s\}$, where \(K^0=S\) and \(G(K^0)=\{1\}\).
\end{lemma}

\begin{proof}
For \(s=0\), the statement is clear. Since
$K=(x_1,\ldots,x_a)(y_1,\ldots,y_b)$, we have $K^s=(x_1,\ldots,x_a)^s(y_1,\ldots,y_b)^s$.
The minimal monomial generators of \((x_1,\ldots,x_a)^s\) are precisely the monomials of degree \(s\) in the variables \(x_1,\ldots,x_a\). Similarly, the minimal monomial generators of \((y_1,\ldots,y_b)^s\) are precisely the monomials of degree \(s\) in the variables \(y_1,\ldots,y_b\). Hence the minimal generators of \(K^s\) are exactly the monomials $x^\gamma y^\beta$ with $|\gamma|=s$, $|\beta|=s$.
\end{proof}

\begin{theorem}\label{thm:ordinary-power-generators}
Let \(I=I(\Gamma_P(R))\). For every integer \(n\geq 1\),
\[
G(I^n)=
\left\{
x^\alpha y^\beta:
|\alpha|+|\beta|=2n,\;
|\beta|\leq n,\;
0\leq \alpha_i\leq n\text{ for all }i=1,\ldots,a
\right\}.
\]
Equivalently, in ring-theoretic parameters,
\[a=|P|-1,\qquad b=|P|(|R/P|-1).
\]
\end{theorem}

\begin{proof}
First let \(x^\alpha y^\beta\in G(I^n)\). Since \(I\) is generated by quadratic
monomials, every minimal generator of \(I^n\) has total degree \(2n\). Hence $|\alpha|+|\beta|=2n$.
Moreover, a \(y\)-variable can appear only in an edge of the form \(x_i y_\ell\).
Each of the \(n\) edge factors contributes at most one \(y\)-variable. Therefore $|\beta|\leq n$. Similarly, a fixed \(x_i\) can appear at most once in each edge factor, so $\alpha_i\leq n$ for every \(i\). Thus every minimal generator of \(I^n\) satisfies the stated
conditions.

Conversely, suppose that \(x^\alpha y^\beta\) satisfies
\[
|\alpha|+|\beta|=2n,\qquad |\beta|\leq n,\qquad \alpha_i\leq n
\]
for all \(i\). Set $s=|\beta|,\qquad r=n-s$. Then \(0\leq s\leq n\) and $|\alpha|=2n-s=2r+s$. We will split \(x^\alpha\) as $x^\alpha=x^\delta x^\gamma$ with $|\delta|=2r,\quad 0\leq \delta_i\leq r$, and $|\gamma|=s$. Then Lemma~\ref{lem:powers-clique-part} gives \(x^\delta\in J^r\), and
Lemma~\ref{lem:powers-cross-part} gives \(x^\gamma y^\beta\in K^s\). Hence $x^\alpha y^\beta=x^\delta x^\gamma y^\beta\in J^rK^s\subseteq (J+K)^n=I^n$.
It remains to construct such a vector \(\gamma\). We need $0\leq \gamma_i\leq \alpha_i$, $|\gamma|=s$, and
$\delta_i=\alpha_i-\gamma_i\leq r$ for every \(i\). Equivalently, $\gamma_i\geq \alpha_i-r$ for every \(i\). Define $\ell_i=\max\{0,\alpha_i-r\}$. We claim that $\ell_1+\cdots+\ell_a\leq s$. Indeed, if at least two coordinates of \(\alpha\) exceed \(r\), the inequality is immediate from \(|\alpha|=2r+s\). If only one coordinate exceeds \(r\), then the inequality follows from \(\alpha_i\leq n=r+s\).
Moreover,
\[
\sum_{i=1}^a(\alpha_i-r)_+
\leq
\sum_{i=1}^a \alpha_i-2r
=
(2r+s)-2r=s,
\]
where \((u)_+=\max\{u,0\}\). Thus $\sum_i \ell_i\leq s$. Choose nonnegative integers \(t_i\) such that $0\leq t_i\leq \alpha_i-\ell_i$ and
\[
\sum_{i=1}^a t_i=s-\sum_{i=1}^a \ell_i.
\]
This is possible because
\[
\sum_{i=1}^a(\alpha_i-\ell_i)
=
|\alpha|-\sum_{i=1}^a\ell_i
\geq
s-\sum_{i=1}^a\ell_i.
\]
Now set $\gamma_i=\ell_i+t_i$. Then $|\gamma|=s$ and \(0\leq \gamma_i\leq \alpha_i\). Also,
$\gamma_i\geq \ell_i\geq \alpha_i-r$, so
$\delta_i=\alpha_i-\gamma_i\leq r$. Furthermore,
$|\delta|=|\alpha|-|\gamma|=(2r+s)-s=2r$.
Therefore \(x^\delta\in J^r\) and \(x^\gamma y^\beta\in K^s\), and so $x^\alpha y^\beta\in I^n$.
Since all monomials satisfying the stated conditions have degree \(2n\), they are
minimal generators of \(I^n\). This completes the proof.
\end{proof}

\begin{example}\label{ex:decomposition-theorem43}
Let \(a=3\) and \(n=2\). Consider the monomial $u=x_1^2x_2y_1$.
Here
\[
\alpha=(2,1,0),\qquad \beta=(1,0,\ldots,0),
\]
so
\[
|\alpha|+|\beta|=3+1=4=2n,\qquad |\beta|=1\leq n,
\]
and each \(\alpha_i\leq n\). Thus \(u\) satisfies the conditions of
Theorem~\ref{thm:ordinary-power-generators}. In the proof, we have
\[
s=|\beta|=1,\qquad r=n-s=1.
\]
We may choose
\[
\delta=(1,1,0),\qquad \gamma=(1,0,0),
\]
so that
\[
\alpha=\delta+\gamma,\qquad |\delta|=2r=2,\qquad |\gamma|=s=1.
\]
Then
\[
x^\delta=x_1x_2\in J
\]
and
\[
x^\gamma y^\beta=x_1y_1\in K.
\]
Therefore
\[
u=x_1^2x_2y_1=(x_1x_2)(x_1y_1)\in JK\subseteq I^2.
\]
This illustrates the splitting \(x^\alpha=x^\delta x^\gamma\) used in the converse
direction of Theorem~\ref{thm:ordinary-power-generators}.
\end{example}

\begin{corollary}\label{cor:star-case-ring}
If \(|P|=2\), then \(a=1\) and $\Gamma_P(R)\cong K_1\vee \overline{K}_{2(q-1)}$. In this case
$I=(xy_1,\ldots,xy_b)$, $b=2(q-1)$, and for every \(n\geq 1\), 
\[G(I^n)=\{x^n y^\beta: |\beta|=n\}.\]
Consequently,
\[
\mu(I^n)=\binom{n+b-1}{b-1}
=
\binom{n+2(q-1)-1}{2(q-1)-1}.
\]
\end{corollary}

\begin{proof}
If \(|P|=2\), then \(a=|P|-1=1\). Hence there is only one \(x\)-variable, say
\(x\). By Theorem~\ref{thm:ordinary-power-generators}, a monomial \(x^\alpha y^\beta\in G(I^n)\) must satisfy $\alpha+|\beta|=2n$, $|\beta|\leq n$, $\alpha\leq n$.
These conditions force \(\alpha=n\) and \(|\beta|=n\). Therefore $G(I^n)=\{x^n y^\beta:|\beta|=n\}$. The number of monomials of degree \(n\) in \(b\) variables is
\[
\binom{n+b-1}{b-1}.
\]
Since \(b=2(q-1)\), the displayed formula follows.
\end{proof}

\begin{theorem}\label{thm:number-generators}
For every integer \(n\geq 1\), the number of minimal monomial generators of
\(I^n\) is
\[
\mu(I^n)
=
\sum_{s=0}^{n}
\binom{s+b-1}{b-1}
\left[
\binom{2n-s+a-1}{a-1}
-
a\binom{n-s+a-2}{a-1}
\right].
\]
Equivalently, using \(a=|P|-1\) and \(b=|P|(|R/P|-1)\), one has
\[
\mu(I^n)
=
\sum_{s=0}^{n}
\binom{s+|P|(|R/P|-1)-1}{|P|(|R/P|-1)-1}
\left[
\binom{2n-s+|P|-2}{|P|-2}
-
(|P|-1)
\binom{n-s+|P|-3}{|P|-2}
\right],
\]
with the convention that \(\binom{u}{v}=0\) whenever \(u<v\) or \(v<0\).
\end{theorem}

\begin{proof}
By Theorem~\ref{thm:ordinary-power-generators}, a minimal generator of \(I^n\) has
the form \(x^\alpha y^\beta\) satisfying
\[
|\alpha|+|\beta|=2n,\qquad |\beta|\leq n,\qquad \alpha_i\leq n
\]
for every \(i\). Fix $s=|\beta|$. Then \(0\leq s\leq n\). The number of possible \(\beta\)'s with \(|\beta|=s\) is
the number of monomials of degree \(s\) in \(b\) variables, namely
\[
\binom{s+b-1}{b-1}.
\]
For a fixed \(s\), we must count the number of vectors $\alpha=(\alpha_1,\ldots,\alpha_a)\in\mathbb{N}^a$ such that $|\alpha|=2n-s$ and $\alpha_i\leq n$ for all \(i\). Without the upper bounds \(\alpha_i\leq n\), the number of
nonnegative solutions of \(|\alpha|=2n-s\) is
\[
\binom{2n-s+a-1}{a-1}.
\]
We subtract the forbidden solutions for which some \(\alpha_i\geq n+1\). Since $2n-s\leq 2n$, at most one component of \(\alpha\) can be at least \(n+1\). Thus, there is no overlap among the forbidden cases. Choose the index \(i\) with \(\alpha_i\geq n+1\).
After replacing $\alpha_i'=\alpha_i-(n+1)$,
we must solve
\[
\alpha_1+\cdots+\alpha_i'+\cdots+\alpha_a
=
2n-s-(n+1)
=
n-s-1.
\]
The number of such solutions is
\[
\binom{n-s-1+a-1}{a-1}
=
\binom{n-s+a-2}{a-1}.
\]
There are \(a\) choices for the index \(i\). Hence the number of admissible
\(\alpha\)'s is
\[
\binom{2n-s+a-1}{a-1}
-
a\binom{n-s+a-2}{a-1}.
\]
Multiplying this by the number of possible \(\beta\)'s and summing over
\(s=0,\ldots,n\) gives
\[
\mu(I^n)
=
\sum_{s=0}^{n}
\binom{s+b-1}{b-1}
\left[
\binom{2n-s+a-1}{a-1}
-
a\binom{n-s+a-2}{a-1}
\right].
\]
Finally, substituting $a=|P|-1$, $b=|P|(|R/P|-1)$ gives the ring-theoretic form of the formula.
\end{proof}

\begin{remark}\label{rem:ordinary-power-ring-data}
Theorem~\ref{thm:number-generators} shows that the ordinary-power sequence $\{\mu(I^n)\}_{n\geq 1}$ is determined entirely by the two finite-ring parameters \(|P|\) and \(|R/P|\).
Thus, if two pairs \((R,P)\) and \((R',P')\) satisfy $|P|=|P'|$ and 
$|R/P|=|R'/P'|$, then the edge ideals of their prime ideal graphs have the same ordinary-power
generator counts.
\end{remark}

\section{The special fiber ring and analytic spread}

In this section, we study the special fiber ring of the edge ideal of a prime ideal
graph. The formula for the number of minimal generators of $I^n$ obtained in Theorem~\ref{thm:number-generators} has a natural interpretation as the Hilbert function of the special fiber ring. We also compute the analytic spread and record structural consequences for the fiber cone.

\begin{proposition}\label{prop:fiber-ring-edge-ring}
Let \(\mathfrak m=(x_1,\ldots,x_a,y_1,\ldots,y_b)\). The special fiber ring $\mathcal{F}(I)$ is isomorphic to the edge ring
\[
\mathcal{F}(I)
\cong
k[x_i x_j,\;x_i y_\ell
\mid 1\leq i<j\leq a,\;1\leq i\leq a,\;1\leq \ell\leq b].
\]
In particular, $\mathcal{F}(I)$ is an affine semigroup ring generated by the
exponent vectors of the edge monomials of $\Gamma_P(R)$.
\end{proposition}

\begin{proof}
Since $I$ is a monomial ideal generated by the edge monomials $x_i x_j \ (1\leq i<j\leq a)$ and $x_i y_\ell\quad (1\leq i\leq a,\ 1\leq \ell\leq b)$, the special fiber ring is generated over $k$ by the images of these monomials
modulo $\mathfrak{m}I$. Therefore $\mathcal{F}(I)$ is naturally isomorphic to
the $k$-subalgebra of $S$ generated by the same edge monomials. Hence
\[
\mathcal{F}(I)
\cong
k[x_i x_j,\;x_i y_\ell
\mid 1\leq i<j\leq a,\;1\leq i\leq a,\;1\leq \ell\leq b].
\]
This is precisely the edge ring of $\Gamma_P(R)$.
\end{proof}

\begin{theorem}\label{thm:analytic-spread}
Let $I=I(\Gamma_P(R))$. Then
\[
\ell(I)=
\begin{cases}
b, & a=1,\\
a+b, & a \geq 2.
\end{cases}
\]
Equivalently, in terms of the finite-ring parameters,
\[
\ell(I)=
\begin{cases}
|R|-2, & |P|=2, \\
|R|-1, & |P|\geq 3.
\end{cases}
\]
Since $|R|=|P||R/P|$, this can also be written as
\[
\ell(I)=
\begin{cases}
2(|R/P|-1), & |P|=2,\\
|P||R/P|-1, & |P|\geq 3.
\end{cases}
\]
\end{theorem}

\begin{proof}
By Proposition~\ref{prop:fiber-ring-edge-ring}, the special fiber ring
$\mathcal{F}(I)$ is the affine semigroup ring generated by the exponent vectors
of the edge monomials
\[
x_i x_j\quad (1\leq i<j\leq a)\quad \text{and}\quad x_i y_\ell\quad (1\leq i\leq a, 1\leq \ell\leq b).
\]
Therefore, the analytic spread $\ell(I)=\dim \mathcal{F}(I)$ is the rank of the subgroup of $\mathbb{Z}^{a+b}$ generated by these exponent
vectors.
First suppose \(a=1\). Then $I=(x_1y_1,\ldots,x_1y_b)$.
The exponent vectors are $e_1+f_1,\ldots,e_1+f_b$,
where $e_1$ corresponds to $x_1$ and $f_\ell$ corresponds to $y_\ell$.
These vectors are linearly independent over $\mathbb{Q}$. Hence $\ell(I)=b$. This case is separated because when \(a=1\) there are no clique edges among the \(x\)-vertices; hence, the exponent vectors span only a \(b\)-dimensional subgroup, not an \((a+b)\)-dimensional one.

Now suppose $a\geq 2$. We show that the exponent vectors span a subgroup of
rank $a+b$. Consider the following $a+b$ vectors:
$e_1+e_2$, 
\[
e_1+f_\ell\qquad (1\leq \ell\leq b)\quad \text{and}\quad e_i+f_1\qquad (2\leq i\leq a).
\]
We prove that these vectors are linearly independent over $\mathbb{Q}$. Suppose
\[
c_0(e_1+e_2)
+
\sum_{\ell=1}^{b} c_\ell(e_1+f_\ell)
+
\sum_{i=2}^{a} d_i(e_i+f_1)
=0.
\]
Comparing the coordinates corresponding to $e_i$ for $i\geq 3$, we get
\[
d_i=0\qquad (i\geq 3).
\]
Comparing the coordinates corresponding to $f_\ell$ for $\ell\geq 2$, we get
\[
c_\ell=0\qquad (\ell\geq 2).
\]
The remaining coordinates give $c_0+c_1=0$, $c_0+d_2=0$, $c_1+d_2=0$.
Hence $c_0=c_1=d_2=0$.
Therefore, all coefficients are zero, and the displayed vectors are linearly
independent. Thus, the exponent vectors of the edge monomials span a subgroup of
rank \(a+b\). Since the polynomial ring has \(a+b\) variables, this is the maximum
possible rank. Hence $\ell(I)=a+b$ when $a\geq 2$.
Finally, substituting $a=|P|-1$, $b=|R|-|P|=|P|(|R/P|-1)$ 
gives $\ell(I)=|R|-2$ when $|P|=2$, and $\ell(I)=|R|-1$ when $|P|\geq 3$. This proves the theorem.
\end{proof}

\begin{corollary}\label{cor:hilbert-function-fiber}
The formula for $\mu(I^n)$ in Theorem~\ref{thm:number-generators} is the Hilbert
function of the special fiber ring $\mathcal{F}(I)$. That is,
\[
\operatorname{HF}_{\mathcal{F}(I)}(n)=\mu(I^n)
\]
for every $n\geq 0$. In particular, for $a\geq 2$, the function $n\longmapsto \mu(I^n)$ agrees for all sufficiently large \(n\) with a polynomial of degree \(a+b-1\). If \(a=1\), it agrees for all sufficiently large \(n\) with a polynomial of degree \(b-1\).
\end{corollary}

\begin{proof}
For any homogeneous ideal generated in one degree, the degree \(n\) component of the special fiber ring satisfies $\mathcal{F}(I)_n=I^n/\mathfrak{m}I^n$. Since $I^n$ is a monomial ideal, the vector-space dimension of $I^n/\mathfrak{m}I^n$ is exactly the number of minimal monomial generators of
$I^n$. Hence
\[
\operatorname{HF}_{\mathcal{F}(I)}(n)
=
\dim_k I^n/\mathfrak{m}I^n
=
\mu(I^n).
\]
The degree of the Hilbert polynomial of $\mathcal{F}(I)$ is
$\dim \mathcal{F}(I)-1=\ell(I)-1$.
The claim therefore follows from Theorem~\ref{thm:analytic-spread}.
\end{proof}

\begin{theorem}\label{thm:fiber-normal-cm}
The special fiber ring $\mathcal{F}(I)$ is a normal affine semigroup ring. Hence
$\mathcal{F}(I)$ is Cohen-Macaulay.
\end{theorem}

\begin{proof}
By Proposition~\ref{prop:fiber-ring-edge-ring}, the special fiber ring $\mathcal{F}(I)$ is the edge ring of the graph $\Gamma_P(R)\cong K_a\vee \overline{K}_b$. By the normality criterion for edge rings of finite simple graphs~\cite{OH, VILL}, the edge ring of a graph is normal if and only if the graph satisfies the odd-cycle condition. The odd-cycle condition says that whenever two odd cycles are vertex
disjoint, there is an edge joining a vertex of one cycle to a vertex of the other.
If $a=1$, then $\Gamma_P(R)\cong K_{1,b}$ is a star graph. Hence, it is bipartite and contains no odd cycle. Therefore, the
odd-cycle condition is satisfied vacuously. Equivalently, in this case
\[
\mathcal{F}(I)\cong k[x_1y_1,\ldots,x_1y_b],
\]
which is a polynomial ring in $b$ variables.
Now assume $a\geq 2$. Let $C_1$ and $C_2$ be two vertex-disjoint odd cycles in
\(\Gamma_P(R)\). Since the vertices corresponding to \(B=R\setminus P\) form an
independent set, every cycle in \(\Gamma_P(R)\) contains at least one vertex from
\(A=P\setminus\{0\}\). Choose
\[
u_i\in C_1\cap A\quad \text{and}\quad u_j\in C_2\cap A.
\]
Since \(C_1\) and \(C_2\) are vertex disjoint, we have \(u_i\neq u_j\). The set
\(A\) is a clique, so \(u_i\) and \(u_j\) are adjacent. Thus, there is an edge joining
a vertex of \(C_1\) to a vertex of \(C_2\). Therefore \(\Gamma_P(R)\) satisfies the
odd-cycle condition.
It follows that the edge ring of \(\Gamma_P(R)\), and hence \(\mathcal{F}(I)\), is
normal. Since \(\mathcal{F}(I)\) is a normal affine semigroup ring, Hochster's
theorem implies that \(\mathcal{F}(I)\) is Cohen-Macaulay.
\end{proof}

\begin{corollary}\label{cor:fiber-summary-ring}
Let \(I=I(\Gamma_P(R))\). Then the fiber-cone invariants of \(I\) depend only on
the two finite-ring parameters \(|P|\) and \(|R/P|\). More precisely,
\[
\operatorname{HF}_{\mathcal{F}(I)}(n)=\mu(I^n)
\]
is given by Theorem~\ref{thm:number-generators},
\[
\ell(I)=
\begin{cases}
|R|-2, & |P|=2,\\
|R|-1, & |P|\geq 3,
\end{cases}
\]
and \(\mathcal{F}(I)\) is normal and Cohen--Macaulay.
\end{corollary}

\begin{proof}
The Hilbert-function statement follows from
Corollary~\ref{cor:hilbert-function-fiber}. The analytic-spread formula follows
from Theorem~\ref{thm:analytic-spread}. Normality and Cohen--Macaulayness follow
from Theorem~\ref{thm:fiber-normal-cm}.
\end{proof}

\begin{remark}\label{rem:fiber-cone-novelty}
The fiber-cone viewpoint strengthens the ordinary-power description. The formula
for \(\mu(I^n)\) is not only a generator count; it is the Hilbert function of the
special fiber ring. Moreover, the analytic spread, normality, and
Cohen-Macaulayness of \(\mathcal{F}(I)\) show that the ordinary powers of
\(I(\Gamma_P(R))\) are controlled by a well-behaved affine semigroup ring determined
by the finite-ring data \(|P|\) and \(|R/P|\).
\end{remark}

\section{Polymatroidal structure and linear powers}

In this section, we prove that the ordinary powers of the edge ideal of a prime
ideal graph have strong homological properties. More precisely, we show that the
edge ideal is matroidal, and hence polymatroidal. It follows that all ordinary
powers are polymatroidal. Consequently, every ordinary power has linear quotients
and a linear minimal free resolution.

We keep the notation
\[
I=I(\Gamma_P(R))\subseteq S=k[x_1,\ldots,x_a,y_1,\ldots,y_b],
\]
where $a=|P|-1$, $b=|P|(|R/P|-1)$. Thus
$\Gamma_P(R)\cong K_a\vee \overline{K}_b$. We assume \(a\geq 1\).

\begin{definition}
Let \(L\subseteq k[z_1,\ldots,z_r]\) be a monomial ideal generated in a single
degree. The ideal \(L\) is called polymatroidal if for any two monomials
\[
u=z_1^{\alpha_1}\cdots z_r^{\alpha_r},
\qquad
v=z_1^{\beta_1}\cdots z_r^{\beta_r}
\]
in \(G(L)\), whenever \(\alpha_i>\beta_i\), there exists an index \(j\) such that
$\alpha_j<\beta_j$
and
\[
\frac{z_j u}{z_i}\in G(L).
\]
If, in addition, \(L\) is square-free and generated in one degree, then \(L\) is
called matroidal when its minimal generators correspond to the bases of a matroid.
Every matroidal ideal is polymatroidal.
\end{definition}

\begin{proposition}\label{prop:edge-ideal-matroidal}
The edge ideal \(I=I(\Gamma_P(R))\) is matroidal. In particular, \(I\) is
polymatroidal.
\end{proposition}

\begin{proof}
Let $X=\{x_1,\ldots,x_a\}$ and $Y=\{y_1,\ldots,y_b\}$. The minimal generators of \(I\) are exactly the square-free quadratic monomials
whose support contains at least one variable from \(X\). Equivalently, the generators of \(I\) correspond to the two-element subsets of \(X\cup Y\) which are not contained entirely in \(Y\). Thus, the corresponding collection of two-element
sets is
\[
\mathcal{B}
=
\{\{z_i,z_j\}\subseteq X\cup Y:\{z_i,z_j\}\not\subseteq Y\}.
\]
We verify the basis-exchange axiom for \(\mathcal{B}\).

Let \(B_1,B_2\in\mathcal{B}\), and let \(e\in B_1\setminus B_2\). We need to find
\(f\in B_2\setminus B_1\) such that
\[
(B_1\setminus\{e\})\cup\{f\}\in\mathcal{B}.
\]
If \(B_2\setminus B_1\) contains an element of \(X\), choose such an element \(f\).
Then $(B_1\setminus\{e\})\cup\{f\}$ contains an element of \(X\), and hence
belongs to \(\mathcal{B}\).
Now suppose that $B_2\setminus B_1\subseteq Y$. Since \(B_2\in\mathcal{B}\), the set \(B_2\) contains at least one element of \(X\). This element cannot lie in \(B_2\setminus B_1\), so it must lie in $B_1\cap B_2$. 
Therefore, after replacing \(e\) by any element \(f\in B_2\setminus B_1\), the set $(B_1\setminus{e})\cup{f}$ still contains an element of \(X\). Hence it belongs to \(\mathcal{B}\).
Thus \(\mathcal{B}\) satisfies the basis-exchange axiom. Hence \(\mathcal{B}\) is
the set of bases of a matroid and \(I\) is a matroidal ideal. Since every
matroidal ideal is polymatroidal, \(I\) is polymatroidal.
\end{proof}

\begin{theorem}\label{thm:powers-polymatroidal}
For every integer \(n\geq 1\), the ordinary power \(I^n\) is polymatroidal.
\end{theorem}

\begin{proof}
By Proposition~\ref{prop:edge-ideal-matroidal}, the ideal \(I\) is polymatroidal. A standard closure property of polymatroidal ideals states that products of polymatroidal ideals are again polymatroidal \cite{HER,CON}. Applying this to the product
\[
I^n=\underbrace{I\cdots I}_{n\text{ times}},
\]
we conclude that \(I^n\) is polymatroidal for every \(n\geq 1\).
\end{proof}

\begin{remark}\label{rem:explicit-generators-and-polymatroidality}
Theorem~\ref{thm:powers-polymatroidal} gives the homological structure of the
ordinary powers, while Theorem~\ref{thm:ordinary-power-generators} gives their
explicit minimal monomial generators. Thus the two results complement each other:
polymatroidality explains why the powers have good homological behavior, whereas
the generator formula gives concrete information about \(\mu(I^n)\), the Hilbert
function of the special fiber ring, and the analytic spread.
\end{remark}

\begin{corollary}\label{cor:linear-quotients}
For every integer \(n\geq 1\), the ideal \(I^n\) has linear quotients.
\end{corollary}

\begin{proof}
By Theorem~\ref{thm:powers-polymatroidal}, the ideal \(I^n\) is polymatroidal. Every polymatroidal ideal has linear quotients with respect to a suitable reverse lexicographic order \cite{HER,JAH}. Therefore
\(I^n\) has linear quotients.
\end{proof}

\begin{corollary}\label{cor:linear-resolution}
For every integer \(n\geq 1\), the ideal \(I^n\) has a \(2n\)-linear minimal free
resolution. In particular,
\[
\operatorname{reg}(I^n)=2n.
\]
Thus \(I=I(\Gamma_P(R))\) has linear powers.
\end{corollary}

\begin{proof}
By Theorem~\ref{thm:ordinary-power-generators}, all minimal monomial generators
of \(I^n\) have degree \(2n\). By Corollary~\ref{cor:linear-quotients}, the ideal
\(I^n\) has linear quotients. A monomial ideal generated in a single degree and
having linear quotients has a linear minimal free resolution. Hence \(I^n\) has a
\(2n\)-linear minimal free resolution. Therefore $\operatorname{reg}(I^n)=2n$.
Since this holds for every \(n\geq 1\), the ideal \(I\) has linear powers.
\end{proof}

\begin{corollary}\label{cor:ring-parameter-linear-powers}
Let \(R\) be a finite commutative ring with identity and let \(P\) be a proper
prime ideal of \(R\) with \(|P|\geq 2\). Then the edge ideal of the prime ideal
graph \(\Gamma_P(R)\) has linear powers. More precisely, for every \(n\geq 1\),
\[
\operatorname{reg}\big(I(\Gamma_P(R))^n\big)=2n.
\]
\end{corollary}

\begin{proof}
This is exactly Corollary~\ref{cor:linear-resolution} applied to
\(I=I(\Gamma_P(R))\).
\end{proof}

\begin{remark}\label{rem:homological-contribution}
The preceding results show that the ordinary powers of \(I(\Gamma_P(R))\) are not
only explicitly computable but also homologically well behaved. In particular, all
ordinary powers are polymatroidal, have linear quotients, and have linear
resolutions. These properties are governed by the complete split structure induced
by the prime ideal \(P\), and the resulting regularity formula is independent of the
individual values of \(|P|\) and \(|R/P|\).
\end{remark}

\begin{proposition}\label{prop:ring-parameter-invariance}
Let \(R\) and \(R'\) be finite commutative rings with identity, and let
\(P\subseteq R\) and \(P'\subseteq R'\) be proper prime ideals. Suppose that
\[
|P|=|P'|
\qquad\text{and}\qquad
|R/P|=|R'/P'|.
\]
Set
\[
I=I(\Gamma_P(R))
\qquad\text{and}\qquad
I'=I(\Gamma_{P'}(R')).
\]
Then \(\Gamma_P(R)\) and \(\Gamma_{P'}(R')\) are isomorphic graphs. Consequently,
after a relabeling of variables, the edge ideals \(I\) and \(I'\) are isomorphic
monomial ideals. In particular, the following invariants agree for \(I\) and \(I'\):

\begin{enumerate}
    \item the number of minimal vertex covers;
    \item the number and heights of the minimal primes;
    \item the ordinary-power generator counts \(\mu(I^n)\) and \(\mu((I')^n)\) for all \(n\geq 1\);
    \item the Hilbert functions of the special fiber rings \(\mathcal{F}(I)\) and \(\mathcal{F}(I')\);
    \item the analytic spreads \(\ell(I)\) and \(\ell(I')\);
    \item the regularity sequences of ordinary powers.
\end{enumerate}

More precisely,
\[
\mu(I^n)=\mu((I')^n)
\]
for all \(n\geq 1\),
\[
\operatorname{HF}_{\mathcal{F}(I)}(n)
=
\operatorname{HF}_{\mathcal{F}(I')}(n)
\]
for all \(n\geq 0\),
\[
\ell(I)=\ell(I'),
\]
and
\[
\operatorname{reg}(I^n)=\operatorname{reg}((I')^n)=2n
\]
for every \(n\geq 1\).
\end{proposition}

\begin{proof}
Set
\[
m=|P|=|P'|
\qquad\text{and}\qquad
q=|R/P|=|R'/P'|.
\]
By Proposition~\ref{prop:ring-induced-split}, we have
\[
\Gamma_P(R)\cong K_{m-1}\vee \overline{K}_{m(q-1)}
\]
and
\[
\Gamma_{P'}(R')\cong K_{m-1}\vee \overline{K}_{m(q-1)}.
\]
Hence
\[
\Gamma_P(R)\cong \Gamma_{P'}(R').
\]
Therefore their edge ideals are isomorphic after a relabeling of variables.

The equality of the numbers of minimal vertex covers and the equality of the
minimal-prime data follow from Proposition~\ref{prop:minimal-vertex-covers-ring}
and Corollary~\ref{cor:primary-decomposition-ring}. The equality of the ordinary-power
generator counts follows from Theorem~\ref{thm:number-generators}. The equality
of the Hilbert functions of the special fiber rings follows from
Corollary~\ref{cor:hilbert-function-fiber}. The equality of analytic spreads follows
from Theorem~\ref{thm:analytic-spread}. Finally, the equality of the regularity
sequences follows from Corollary~\ref{cor:linear-resolution}. Hence all listed
invariants depend only on the pair $(|P|,\ |R/P|)$ and not on the particular finite ring \(R\) or the particular prime ideal \(P\).
\end{proof}
\begin{remark}
Proposition~\ref{prop:ring-parameter-invariance} shows that the algebraic invariants studied in this paper are invariants of the finite-ring parameter pair \((|P|,|R/P|)\), rather than of a particular presentation of the ring.
\end{remark}
\section{Examples}

In this section we illustrate the preceding results for some finite commutative rings.
The examples show how the parameters
\[
m=|P|,\qquad q=|R/P|,\qquad a=m-1,\qquad b=m(q-1)
\]
determine the prime ideal graph, the edge ideal, and the ordinary-power generator
counts.

\begin{example}
Let \(R=\mathbb{Z}_6\) and let \(P=(3)=\{0,3\}\). Then
\[
|P|=2,\qquad |R/P|=3.
\]
Hence
\[
m=2,\qquad q=3,\qquad a=m-1=1,\qquad b=m(q-1)=4.
\]
Therefore
\[
\Gamma_P(R)\cong K_1\vee \overline{K}_4.
\]
Thus \(\Gamma_P(R)\) is a star graph on five vertices. If
\[
S=k[x,y_1,y_2,y_3,y_4],
\]
then
\[
I(\Gamma_P(R))=(xy_1,xy_2,xy_3,xy_4).
\]
By Corollary~\ref{cor:star-case-ring},
\[
\mu(I^n)=\binom{n+3}{3}.
\]
In particular,
\[
\mu(I^2)=\binom{5}{3}=10.
\]
\end{example}

\begin{example}
Let \(R=\mathbb{Z}_8\) and let \(P=(2)=\{0,2,4,6\}\). Then
\[
|P|=4,\qquad |R/P|=2.
\]
Hence
\[
m=4,\qquad q=2,\qquad a=m-1=3,\qquad b=m(q-1)=4.
\]
Therefore
\[
\Gamma_P(R)\cong K_3\vee \overline{K}_4.
\]
The corresponding edge ideal is
\[
I=(x_i x_j,\;x_i y_\ell
\mid 1\leq i<j\leq 3,\;1\leq i\leq 3,\;1\leq \ell\leq 4).
\]
Using Theorem~\ref{thm:number-generators} with \(n=2\), \(a=3\), and \(b=4\), we obtain
\[
\mu(I^2)=94.
\]
\end{example}

\begin{example}
Let
\[
R=\mathbb{F}_3\times \mathbb{Z}_4
\]
and let
\[
P=\{0\}\times \mathbb{Z}_4.
\]
Then \(P\) is a prime ideal of \(R\), because
\[
R/P\cong \mathbb{F}_3.
\]
Moreover,
\[
|P|=4,\qquad |R/P|=3.
\]
Thus
\[
m=4,\qquad q=3,\qquad a=m-1=3,\qquad b=m(q-1)=8.
\]
Hence
\[
\Gamma_P(R)\cong K_3\vee \overline{K}_8.
\]
This example illustrates the realization theorem: the complete split graph
\(K_3\vee \overline{K}_8\) arises from a prime ideal of a finite commutative ring.

Using Theorem~\ref{thm:number-generators} with \(n=2\), \(a=3\), and \(b=8\), we obtain
\[
\begin{aligned}
\mu(I^2)
&=
\sum_{s=0}^{2}
\binom{s+7}{7}
\left[
\binom{4-s+2}{2}
-
3\binom{2-s+1}{2}
\right]  \\
&=
1(15-9)+8(10-3)+36(6-0) \\
&=
6+56+216 \\
&=
278.
\end{aligned}
\]
\end{example}
Table~\ref{tab:ring-parameters} summarizes the ring parameters and the derived invariants for the three worked examples.
\begin{table}[h]
\centering
\caption{Ring parameters, generator counts, and homological invariants.}
\label{tab:ring-parameters}
\begin{tabular}{c c c c c c c c c}
\hline
Ring \(R\) & Prime ideal \(P\) & \(m\) & \(q\) & \(a\) & \(b\) &
\(\mu(I^2)\) & \(\ell(I)\) & \(\operatorname{reg}(I^2)\) \\
\hline
\(\mathbb{Z}_6\) & \((3)\) & \(2\) & \(3\) & \(1\) & \(4\) & \(10\) & \(4\) & \(4\) \\
\(\mathbb{Z}_8\) & \((2)\) & \(4\) & \(2\) & \(3\) & \(4\) & \(94\) & \(7\) & \(4\) \\
\(\mathbb{F}_3\times \mathbb{Z}_4\) & \(\{0\}\times \mathbb{Z}_4\) & \(4\) & \(3\) & \(3\) & \(8\) & \(278\) & \(11\) & \(4\) \\
\hline
\end{tabular}
\end{table}

\section{Conclusion}

In this paper we studied ordinary powers and fiber-cone invariants of edge ideals arising from prime ideal graphs of finite commutative rings. The first main result gives a ring-theoretic classification of the underlying graphs. If \(P\) is a proper prime ideal of a finite commutative ring \(R\), and $m=|P|$, $q=|R/P|$,
then \(q\) is a prime power and
$\Gamma_P(R)\cong K_{m-1}\vee \overline{K}_{m(q-1)}$.
Conversely, every complete split graph of this form is realized as the prime ideal graph of a suitable finite commutative ring. Thus, the graph family considered here is not arbitrary; it is determined by the finite-ring pair \((|P|,|R/P|)\).
Using this classification, we determined the minimal vertex covers and the irredundant primary decomposition of \(I=I(\Gamma_P(R))\). We then described the minimal monomial generators of every ordinary power \(I^n\) and obtained a closed formula for \(\mu(I^n)\). These formulas show that the ordinary-power generator sequence depends only on the two ring-theoretic parameters \(|P|\) and \(|R/P|\).
We also studied the special fiber ring \(\mathcal{F}(I)\). The generator-count formula was interpreted as the Hilbert function of \(\mathcal{F}(I)\), and the analytic spread was computed explicitly. Moreover, by identifying \(\mathcal{F}(I)\) with the edge ring of \(\Gamma_P(R)\), we proved that \(\mathcal{F}(I)\) is a normal Cohen-Macaulay affine semigroup ring. Finally, we showed that \(I\) is matroidal and that all ordinary powers \(I^n\) are polymatroidal. Hence \(I^n\) has linear quotients and a \(2n\)-linear minimal free resolution for every \(n\geq 1\).
These results provide a uniform ordinary-power, fiber-cone, and homological analysis for a ring-realizable family of complete split graphs. Possible further directions include the study of graded Betti numbers, Rees algebra invariants, depth functions, and related edge ideals arising from other graph constructions over finite commutative rings.

\section*{Statements and Declarations}

\subsection*{Funding} 
The authors received no specific funding for this work.
\subsection*{Competing interests}
The authors declare that they have no competing interests.
\subsection*{Data availability}
No datasets were generated or analyzed during the current study.

\end{document}